\documentclass[11pt,a4paper]{article}
\usepackage[T1]{fontenc}
\usepackage{lmodern}
\usepackage[margin=28mm]{geometry}
\usepackage{amsmath,amssymb,amsthm,mathtools}
\usepackage{microtype}
\newcommand{\rev}[1]{{#1}}
\newenvironment{revision}{\par\begingroup}{\par\endgroup}
\usepackage[hidelinks]{hyperref}
\numberwithin{equation}{section}
\newtheorem{theorem}{Theorem}[section]
\newtheorem{lemma}[theorem]{Lemma}
\newtheorem{proposition}[theorem]{Proposition}
\newtheorem{corollary}[theorem]{Corollary}
\newtheorem{conjecture}[theorem]{Conjecture}
\theoremstyle{definition}

\newtheorem{problem}[theorem]{Problem}
\theoremstyle{remark}
\newtheorem{remark}[theorem]{Remark}
\newtheorem{example}[theorem]{Example}
\newcommand{\RR}{\mathbb R}
\newcommand{\tN}{\widetilde N}
\newcommand{\tP}{\widetilde{\mathcal P}}
\newcommand{\tPsi}{\widetilde\Psi}
\newcommand{\Sym}{\operatorname{Sym}}
\newcommand{\HH}{\mathcal H}
\newcommand{\bm}{\mathbf m}
\newcommand{\balpha}{\boldsymbol\alpha}

\newcommand{\cM}{\mathcal M}
\newcommand{\cK}{\mathcal K}
\hypersetup{pdftitle={Quasirandomness from balanced and repeated subgraph counts},
  pdfauthor={Guoqing Wang, Linlin Wang, Jun Yue}}

\title{Quasirandomness from balanced and repeated subgraph counts}
\author{{\textsc{Guoqing Wang, Linlin Wang, and Jun Yue}}}
\date{}

\begin{document}
\maketitle

\begin{abstract}

{
Let $F$ be a graph on $m$ vertices with at least one edge,
and let $0<p\le1$. We study when prescribed subgraph counts
across disjoint vertex classes characterize $p$-quasirandom
graph sequences. For positive multiplicities
$\bm=(m_1,\ldots,m_r)$ summing to $m$ and positive class
proportions $\balpha=(\alpha_1,\ldots,\alpha_r)$ with total
at most one, we first classify symmetrized repeated-class
counts for $r\ge3$: when $0<p<1$, they are forcing precisely
if $e(F)>1$ or $\balpha\ne\bm/m$. We also obtain a complete
classification of labelled transversal counts: for
\rev{$0<p<1$, the only nonforcing profile is $F=K_2$
with class proportions $(1/2,1/2)$.} Together,
these results settle conjectures of Janson and S\'os
[European J. Combin. 46 (2015)], Hatami, Hatami and Li
[European J. Combin. 56 (2016)], and Janson [arXiv:2606.20186v1 (2026)].
The common ingredient is a balanced-count theorem for every
graph with more than one edge, proved through a degree-sequence
inequality, and its consequence of additive rigidity.
Our arguments also determine the bad indices of the
associated functional nullspaces and isolate the exceptional
single-edge profiles. In particular, a non-divisibility
theorem for grouped elementary symmetric polynomials
excludes the higher-order obstructions in the repeated-class
setting. In contrast, for two repeated classes and every
$p\in(0,1)$, we construct nonquasirandom graph sequences
satisfying the crossing-triangle counts at proportions
$(1/3,2/3)$, with limiting edge density strictly below $p$.
These counterexamples answer the crossing-triangle question
negatively and show why the three-class assumption cannot
be omitted.
}

\end{abstract}


\par

\section{Introduction}\label{sec:introduction}

The work of Thomason~\cite{Thomason} and Chung, Graham and
Wilson~\cite{CGW} established the equivalence of several
asymptotic counting and spectral conditions for dense graph
sequences. Simonovits and S\'os~\cite{SS} subsequently proved
that the counts of any fixed graph containing an edge determine
quasirandomness when tested in every vertex subset. More
precisely, let $F$ have $m$ vertices and $e(F)>0$, let
$0<p\le1$, and let $v_n=v(G_n)\to\infty$. Then the identities
$N(F,G_n;U^m)=p^{e(F)}|U|^m+o(v_n^m)$, uniformly over
$U\subseteq V(G_n)$, characterize $p$-quasirandomness.
Conlon, Fox and Sudakov~\cite{CFS} later obtained polynomial
quantitative bounds for this hereditary counting condition.

A natural question is whether quasirandomness can also be
characterized by counts across disjoint vertex classes of
prescribed sizes. We consider this problem in a setting that
allows several vertices of $F$ to lie in the same class.
The properties $\mathcal P^p_{\bm}(F;\balpha)$ and
$\tP^p_{\bm}(F;\balpha)$ retain a fixed labelling of $F$
and average over all labellings, respectively. The formal
definitions, including the uniform error terms and integer
rounding, are given in Section~\ref{sec:preliminaries}.
When each vertex of $F$ is assigned to a distinct class,
we call the counts transversal and omit the subscript $\bm$.
\rev{We write $\tN$ for the average of labelled counts over all
assignments of the vertices of $F$ to the prescribed sets.
We call class proportions \emph{complementary} when their sum is one.}
Our aim is to determine which of these restricted counting
properties characterize $p$-quasirandomness.

Huang and Lee~\cite{HL} proved that balanced transversal counts
characterize quasirandomness for complete graphs on at least
three vertices. Janson and S\'os~\cite[Theorems~2.11 and~2.12]{JS}
showed that, for every graph containing an edge, symmetrized
transversal counts characterize quasirandomness whenever the
class proportions are not balanced. \rev{In the balanced case,
they established the same conclusion for graphs with at least two
edges that are regular, are stars, or have at least two components
containing edges.} They
conjectured that the conclusion holds for every graph with at
least two edges, as stated below.

\begin{conjecture}[Janson--S\'os
{\cite[Conjecture~2.13]{JS}}]\label{conj:JS-balanced}
Let $F$ be a graph on $m$ vertices with $e(F)>1$, and let
$0<p\le1$. If a graph sequence $(G_n)$ with
$v_n=v(G_n)\to\infty$ satisfies
\[
\tN(F,G_n;U_1,\ldots,U_m)
=p^{e(F)}\prod_{i=1}^m|U_i|+o(v_n^m)
\]
uniformly over all pairwise disjoint sets
$U_1,\ldots,U_m\subseteq V(G_n)$ with
$|U_i|=\lfloor v_n/m\rfloor$, then $(G_n)$ is
$p$-quasirandom.
\end{conjecture}

{
The main results of this paper give two classifications of restricted-count
conditions. The first concerns symmetrized counts in which a class may
be repeated; the second concerns labelled transversal counts. We state
them first, before describing the balanced-count theorem and the
rigidity principle used in their proofs.
} {
Janson~\cite{J26} introduced repeated-class profiles, in which
several vertices of $F$ may be assigned to the same class. For at least
three classes, the following theorem gives the full classification.
}

\begin{theorem}\label{thm:repeated}
Let $F$ be a graph on $m$ vertices with $e(F)>0$, let
$\bm=(m_1,\ldots,m_r)$ be a vector of positive integers with
$\sum_{j=1}^r m_j=m$, and let
$\balpha=(\alpha_1,\ldots,\alpha_r)$ have positive entries
with $\sum_{j=1}^r\alpha_j\le1$. If $r\ge3$ and $0<p<1$,
then $\tP^p_{\bm}(F;\balpha)$ is a $p$-quasirandom property
if and only if
\[
e(F)>1\qquad\text{or}\qquad\balpha\ne\bm/m.
\]
\end{theorem}

Thus every graph with at least two edges is forcing for
every admissible profile with at least three classes.
For a graph consisting of one edge and any number of
isolated vertices, the condition fails to characterize
$p$-quasirandomness precisely when $\alpha_j=m_j/m$ for
every $j$. Taking $r=m$ and $m_1=\cdots=m_m=1$ recovers
the symmetrized assertion of Theorem~\ref{thm:balanced}
for $0<p<1$. The endpoint $p=1$ is treated separately
in Lemma~\ref{lem:p-one}, which applies to every graph
containing an edge and every admissible profile.

{
The second classification concerns labelled transversal
counts. In this setting, Hatami, Hatami and Li~\cite{HHL}
characterized the functions whose integrals vanish over
products of disjoint sets of prescribed measures. They
used this characterization to prove that labelled
transversal counts characterize quasirandomness for
graphs with twin vertices, and proposed the following
conjecture.
}

\begin{conjecture}[Hatami--Hatami--Li
{\cite[Conjecture~6.1]{HHL}}]\label{conj:HHL}
Let $F\ne K_2$ be a labelled graph on $m$ vertices with
$e(F)>0$, let $0<p\le1$, and let $\alpha_i>0$ satisfy
$\sum_{i=1}^m\alpha_i\le1$. Then
$\mathcal P^p(F;\balpha)$ is a $p$-quasirandom property.
\end{conjecture}

{
We establish the following classification, including the
classical case $F=K_2$.
}

\begin{theorem}\label{thm:transversal}
Let $F$ be a labelled graph on $m$ vertices with $e(F)>0$,
let $0<p<1$, and let $\alpha_i>0$ with
$\sum_{i=1}^m\alpha_i\le1$. Then
$\mathcal P^p(F;\balpha)$ is a $p$-quasirandom property
unless $F=K_2$ and $\alpha_1=\alpha_2=1/2$.
In that exceptional case, it is not a $p$-quasirandom property.
\end{theorem}

Together with Lemma~\ref{lem:p-one}, this proves
Conjecture~\ref{conj:HHL} for $0<p\le1$ and answers
\cite[Problem~2.16]{JS} affirmatively.
In particular, a graph consisting of one edge and at least
one isolated vertex is forcing for every labelled transversal
profile, although its balanced symmetrized condition fails
for $0<p<1$. Since the exceptional $K_2$ test admits a
two-type graphon counterexample, the classification also
answers \cite[Problems~2.15 and~2.17]{JS} affirmatively.

{
The common ingredient in these classifications is the resolution
of the balanced transversal-count conjecture of Janson and
S\'os~\cite{JS}, stated above. The theorem includes labelled counts
as well as their symmetrized counterparts.
}

\begin{theorem}\label{thm:balanced}
Let $F$ be a graph on $m$ vertices with $e(F)>1$, and let
$0<p\le1$. Then both
$\tP^p(F;1/m,\ldots,1/m)$ and
$\mathcal P^p(F;1/m,\ldots,1/m)$ are
$p$-quasirandom properties.
\end{theorem}

{
Theorem~\ref{thm:balanced} requires no connectedness or
regularity assumption and allows isolated vertices. Its proof completes
the algebraic reduction of Janson and S\'os by a degree-sequence
inequality. The resulting additive rigidity statement
(Corollary~\ref{cor:additive}) provides a common ingredient for both
classifications.
}

\begin{revision}
The arguments underlying Theorem~\ref{thm:repeated} also
settle two conjectures of Janson~\cite{J26}. The first,
\cite[Conjecture~5.20]{J26}, concerns the functional nullspace
of the restricted averages. We show that, with at least three
classes, this nullspace contains only mean-zero additive
functions: all higher Walsh levels are excluded by the
non-divisibility theorem for grouped elementary symmetric
polynomials (Theorem~\ref{thm:app520-factor}). The definitions
are given in Section~\ref{sec:preliminaries}.
The second, \cite[Conjecture~6.3]{J26}, asserts that profiles
whose nullspace consists exactly of the mean-zero additive
functions are forcing whenever $e(F)>1$.
Corollary~\ref{cor:additive-profile} proves the
more general assertion that, for $e(F)>1$ and $0<p\le1$,
forcing holds whenever every function in the nullspace is
mean-zero and additive, regardless of the number of classes.
\end{revision}

For two classes, however, higher Walsh levels may survive,
and the conclusion of Theorem~\ref{thm:repeated} need not
hold. This provides an obstruction to the following
conjecture of Janson, which extends the forcing assertion
to all repeated-class profiles with more than one class.

\begin{conjecture}[Janson
{\cite[Conjecture~6.7]{J26}}]\label{conj:J-general}
Let $F$ have at least two edges, let $0<p<1$, and let
$(\bm,\balpha)$ be an admissible profile with $r>1$.
Then $\tP^p_{\bm}(F;\balpha)$ is a $p$-quasirandom property.
\end{conjecture}

A particular case is the crossing-triangle question,
corresponding to $F=K_3$, $\bm=(2,1)$ and
$\balpha=(1/3,2/3)$, posed in \cite[Problem~9.5]{JS}
and restated in \cite[Problem~6.8]{J26}.
\rev{Here the counted triangles have exactly two vertices in the
class of proportion $1/3$ and one in the class of proportion $2/3$.}
Since the labelled and symmetrized triangle counts coincide,
this case concerns a single counting condition.
The following proposition gives a counterexample for every
$p\in(0,1)$.

\begin{proposition}\label{prop:crossing}
For every $0<p<1$, there is a graph sequence $(G_n)$
with $v_n=v(G_n)\to\infty$ such that
\[
N(K_3,G_n;U,U,V)=p^3|U|^2|V|+o(v_n^3)
\]
uniformly over all disjoint sets $U,V\subseteq V(G_n)$
with $|U|=\lfloor v_n/3\rfloor$ and
$|V|=\lfloor2v_n/3\rfloor$, while
\[
\lim_{n\to\infty}\frac{2e(G_n)}{v_n^2}<p.
\]
In particular, $(G_n)$ is not $p$-quasirandom.
\end{proposition}

Proposition~\ref{prop:crossing} answers the crossing-triangle
question negatively and disproves
Conjecture~\ref{conj:J-general}. The construction realizes
a nonzero discrepancy in $\HH_{3,2}$ whose integral vanishes
over every prescribed product of classes. At
$p=4^{-1/3}$, the sequence $K_{3t}\sqcup K_{3t}$ provides
a direct finite example. These counterexamples have limiting
edge density different from $p$ and therefore do not settle
the variant in which edge density $p$ is also required.
The general classification for two complementary classes
remains open.

{
\noindent\textit{Outline of the proofs.} The distinction between the classification results and the
two-class counterexamples arises from the different functional
nullspaces of the restricted-count conditions. The graphon transfer
principle reduces forcing to the constancy of a subgraph kernel.
For balanced counts, the Janson--S\'os algebraic reduction yields
a triple $(U,V,S)$ satisfying a system of equations. The four equations
indexed by $0,1,m-1,m$, together with a scalar inequality for the
degree sequence of $F$, force this triple to be constant. This proves
Theorem~\ref{thm:balanced} and, through the balanced-count condition,
the additive rigidity needed in the two classifications.
}

{
For symmetrized repeated-class counts, Janson's nullspace
characterization expresses the discrepancy in terms of bad Walsh
levels. When $r\ge3$, the grouped elementary symmetric polynomials
\rev{have no linear divisor whose coefficients are all positive
at levels $k\ge2$.} Thus only the additive level survives, and additive rigidity
settles the case $e(F)>1$; the single-edge case determines the exception.
For labelled transversal counts, the characterization of Hatami,
Hatami and Li gives direct vanishing for noncomplementary proportions;
for complementary proportions, higher Walsh components vanish under
symmetrization. This again gives an additive reduction. With two
repeated classes, a second-order component may survive. The crossing-triangle
counterexamples realize this obstruction by a two-type graphon.
}

{
The paper is organized as follows. Section~\ref{sec:preliminaries}
introduces the restricted counts, graphons and the two functional
characterizations. Section~\ref{sec:balanced} proves the balanced-count
theorem and additive rigidity. Section~\ref{sec:repeated} establishes
polynomial non-divisibility, classifies the bad indices and proves the
symmetrized repeated-class classification. Section~\ref{sec:transversal}
settles the labelled transversal case. Finally,
Section~\ref{sec:counterexamples} constructs the two-class
counterexamples and states the remaining classification problem.
}

\section{Preliminaries}\label{sec:preliminaries}

\begin{revision}
We fix the counting conventions and class profiles, then
record the graphon transfer principle and the functional
characterizations used below.
\end{revision}

\subsection{Graphs and restricted subgraph counts}

All graphs considered in this paper are finite, undirected and simple. 
For a graph $G=(V(G),E(G))$, put $v(G)=|V(G)|$ and $e(G)=|E(G)|$.
Let $[m]=\{1,\ldots,m\}$. Throughout, $F$ is a fixed graph
with vertex set $[m]$ and $e=e(F)>0$. For subsets $U_1,\ldots,U_m\subseteq V(G)$, let
$N(F,G;U_1,\ldots,U_m)$ denote the number of injective maps
$\phi:[m]\to V(G)$ such that $\phi(i)\in U_i$ for every
$i\in[m]$ and $\{\phi(i),\phi(j)\}\in E(G)$ whenever
$\{i,j\}\in E(F)$. Thus $N$ counts labelled copies of $F$
with prescribed vertex memberships; these copies need not
be induced. To average over the assignments of the labels to the
prescribed sets, let $\mathfrak S_m$ be the group of all
permutations of $[m]$ and define
\[
\tN(F,G;U_1,\ldots,U_m)
=\frac{1}{m!}\sum_{\pi\in\mathfrak S_m}
N(F,G;U_{\pi(1)},\ldots,U_{\pi(m)}).
\]
The sets $U_i$ need not be disjoint or distinct.
For a positive integer $a$, the notation $U^a$ in an
argument list denotes $a$ consecutive occurrences of $U$.
In particular, if $m_1,\ldots,m_r$ are positive integers
with $\sum_{j=1}^r m_j=m$, then
$N(F,G;U_1^{m_1},\ldots,U_r^{m_r})$ denotes the count
obtained by repeating each $U_j$ exactly $m_j$ times.
The same convention applies to $\tN$.

Let $(G_n)$ be a graph sequence with $v_n=v(G_n)\to\infty$.
For $U\subseteq V(G_n)$, let $G_n[U]$ denote the subgraph
induced by $U$. Given $p\in[0,1]$, we call $(G_n)$
\emph{$p$-quasirandom} if
$e(G_n[U])=p|U|^2/2+o(v_n^2)$ uniformly over
$U\subseteq V(G_n)$. Fix $1\le r\le m$, a vector $\bm=(m_1,\ldots,m_r)$ of
positive integers with $\sum_{j=1}^r m_j=m$, and a vector
$\balpha=(\alpha_1,\ldots,\alpha_r)$ of positive reals with
$\sum_{j=1}^r\alpha_j\le1$. We call $(\bm,\balpha)$ an
\emph{admissible profile}, where $\bm$ and $\balpha$ specify
the class multiplicities and proportions, respectively.
All these parameters are independent of $n$. The sequence
$(G_n)$ satisfies the \emph{labelled restricted-count property}
$\mathcal P^p_{\bm}(F;\balpha)$ if
\[
N(F,G_n;U_1^{m_1},\ldots,U_r^{m_r})
=p^e\prod_{j=1}^r|U_j|^{m_j}+o(v_n^m)
\]
uniformly over all pairwise disjoint sets
$U_1,\ldots,U_r\subseteq V(G_n)$ with
$|U_j|=\lfloor\alpha_jv_n\rfloor$ for every $j\in[r]$.
Replacing $N$ by $\tN$ defines the \emph{symmetrized
restricted-count property} $\tP^p_{\bm}(F;\balpha)$.
Uniformity means that the absolute error is at most
$\varepsilon_n v_n^m$ for some $\varepsilon_n\to0$
independent of the permitted vertex sets. When $r=m$, necessarily $m_1=\cdots=m_m=1$; we then omit the subscript $\bm$ and call the counts \emph{transversal}.
A profile is \emph{balanced} if $\balpha=\bm/m$, that is,
$\alpha_j=m_j/m$ for every $j\in[r]$. A count property is
a \emph{$p$-quasirandom property}, or \emph{forcing} for
fixed $p$, if it holds precisely for the $p$-quasirandom
graph sequences.

For $T\subseteq V(F)$, write $T^c=V(F)\setminus T$,
put $e_F(T)=e(F[T])$, and let $e_F(T,T^c)$ denote the
number of edges with one endpoint in each part.
For $i\in[m]$, write $d_i=d_F(i)$ for its degree in $F$.
A factor with exponent zero is interpreted as $1$, including
when its base is zero. For a nonnegative integer $a$,
the notation $u^{[a]}$ denotes a list of $a$ copies of $u$,
with $u^{[0]}$ denoting the empty list.

\subsection{Graphons and the transfer principle}

A \emph{graphon} is a symmetric measurable function
$W:[0,1]^2\to[0,1]$. Write $\lambda$ for Lebesgue measure
on $[0,1]$ and its product spaces. For a graph $F$ on $[m]$,
define the labelled and symmetrized subgraph kernels by
\[
\Psi_{F,W}(x_1,\ldots,x_m)
=\prod_{ij\in E(F)}W(x_i,x_j),
\qquad
\tPsi_{F,W}=\Sym\Psi_{F,W},
\]
where, for a measurable function $f$ on $[0,1]^m$,
\[
\Sym f(x_1,\ldots,x_m)
=(m!)^{-1}\sum_{\pi\in\mathfrak S_m}
f(x_{\pi(1)},\ldots,x_{\pi(m)})\,.
\]
Equalities between measurable functions are understood
almost everywhere.

For a real integrable kernel $K$ on $[0,1]^2$, define
its \emph{cut norm} by
\[
\|K\|_\square
=\sup_{B,C\subseteq[0,1]}
\left|\int_{B\times C}K(x,y)\,dx\,dy\right|,
\]
where the supremum is over all measurable sets $B,C$.
The \emph{cut distance} between graphons $U,V$ is
$\delta_\square(U,V)
=\inf_\varphi\|U-V^\varphi\|_\square$, where
$V^\varphi(x,y)=V(\varphi(x),\varphi(y))$ and the infimum
is over measure-preserving bijections $\varphi$ of $[0,1]$,
understood modulo null sets.

To associate a graphon with a finite graph $G$ on
$N\ge1$ vertices, enumerate its vertices as
$v_1,\ldots,v_N$ and partition $[0,1]$ into measurable
cells $I_1,\ldots,I_N$, each of measure $1/N$.
Set $W_G(x,y)=1$ on $I_i\times I_j$ if
$v_iv_j\in E(G)$, and set $W_G(x,y)=0$ otherwise.
The choice of vertex ordering and cells does not affect
its cut-distance equivalence class. A graph sequence
$(G_n)$ with $v(G_n)\to\infty$ is $p$-quasirandom if
and only if $\delta_\square(W_{G_n},p)\to0$, where $p$
also denotes the constant graphon; see~\cite{LS,J11}.
This characterization allows us to translate restricted
counting conditions into integral identities for graphons.
We first record the standard counting property of
quasirandom graph sequences.

\begin{lemma}[Uniform restricted counting
{\cite[Lemma~2.4]{JS}}]\label{lem:uniform-counting}
Let $(G_n)$ be a $p$-quasirandom graph sequence with
$v_n=v(G_n)\to\infty$, where $p\in[0,1]$, and let $F$
be a fixed graph on $[m]$ with $e=e(F)>0$. Then
\[
N(F,G_n;U_1,\ldots,U_m)
=p^e\prod_{i=1}^m|U_i|+o(v_n^m)
\]
uniformly over all subsets
$U_1,\ldots,U_m\subseteq V(G_n)$.
The same conclusion holds with $\tN$ in place of $N$.
\end{lemma}

The sets in Lemma~\ref{lem:uniform-counting} need not be
disjoint or distinct. Thus the lemma applies to repeated
classes and shows that every $p$-quasirandom sequence
satisfies each of the restricted-count conditions considered
here. To prove that such a condition is forcing, it therefore
suffices to establish the converse implication. The following
transfer principle expresses this implication in terms of
graphon integral identities.

\begin{lemma}[Restricted-count transfer
{\cite[Lemma~3.3]{J26}}]\label{lem:transfer}
Let $F$ be a graph on $[m]$ with $e=e(F)>0$, let
$(\bm,\balpha)$ be an admissible profile, and let $0<p<1$.
Then $\mathcal P^p_{\bm}(F;\balpha)$ is a
$p$-quasirandom property if and only if every graphon $W$
satisfying
\[
\int_{A_1^{m_1}\times\cdots\times A_r^{m_r}}
\Psi_{F,W}(\mathbf{x})\,d\mathbf{x}
=p^e\prod_{j=1}^r\alpha_j^{m_j}
\]
for all pairwise disjoint measurable sets
$A_1,\ldots,A_r\subseteq[0,1]$ with
$\lambda(A_j)=\alpha_j$ is equal to $p$ almost everywhere.
Here $\mathbf{x}=(x_1,\ldots,x_m)$ and
$d\mathbf{x}=dx_1\cdots dx_m$.
The same equivalence holds for
$\tP^p_{\bm}(F;\balpha)$ with $\tPsi_{F,W}$ in place
of $\Psi_{F,W}$.

More precisely, suppose that $v(G_n)\to\infty$ and
$\delta_\square(W_{G_n},W)\to0$. Then $(G_n)$ satisfies
$\mathcal P^p_{\bm}(F;\balpha)$ if and only if $W$
satisfies the displayed integral identities. The analogous
statement holds for symmetrized counts and kernels.
\end{lemma}

\begin{revision}
For transversal counts, the corresponding transfer results
are given in \cite[Lemmas~3.1 and~3.4]{JS}.
The stronger formulation in Lemma~\ref{lem:transfer} uses
the same partition-approximation argument as these results.
The relevant uniform estimate is
\[
\left|
\int_{B_1\times\cdots\times B_m}
\bigl(\Psi_{F,U}-\Psi_{F,V}\bigr)
\right|
\le e(F)\|U-V\|_\square,
\]
valid for graphons $U,V$ and arbitrary measurable sets
$B_1,\ldots,B_m$, including repeated sets. To see this,
telescope over the edges of $F$ and fix the remaining
coordinates; each summand is controlled by the cut norm
with two weights in $[0,1]$. The same bound holds after
symmetrization. For a graph on $N$ vertices, excluding
noninjective maps changes the count by $O(N^{m-1})$,
uniformly in the vertex classes. Changing each class by
$O(1)$ vertices has the same order of effect. These estimates,
together with the standard partition-approximation argument
in the cited transfer results, give the uniform errors and
integer rounding in the statement.

Thus the forcing problem reduces to showing that the
prescribed integral identities imply $W=p$ almost everywhere.
A nonconstant graphon satisfying the identities yields
a graph sequence with the restricted counts that is not
$p$-quasirandom.
\end{revision}

A graphon has \emph{two types} if there is a measurable
partition $[0,1]=B_1\cup B_2$, with
$\lambda(B_1),\lambda(B_2)>0$, such that $W$ is constant
almost everywhere on each $B_i\times B_j$, for
$i,j\in\{1,2\}$. The counterexamples constructed below
have this form. Before introducing the functional
characterizations used to prove forcing, we treat the
endpoint $p=1$, which is excluded from the standing
assumptions of~\cite{J26}.

\begin{lemma}\label{lem:p-one}
Let $F$ be a simple graph on $[m]$ with $e(F)>0$, and let
$(\bm,\balpha)$ be an admissible profile. Every graph
sequence satisfying $\mathcal P^1_{\bm}(F;\balpha)$
or $\tP^1_{\bm}(F;\balpha)$ is $1$-quasirandom.
\end{lemma}

\begin{proof}
Let $(G_n)$ satisfy either condition, put $N=v(G_n)$,
and set $k_j=\lfloor\alpha_jN\rfloor$ for $j\in[r]$.
Since every $\alpha_j>0$ and the multiplicities are fixed,
we may assume that $k_j\ge m_j$ for every $j$.
Moreover, $\sum_jk_j\le N$.

Choose a uniformly random ordering of $V(G_n)$ and form
$U_1,\ldots,U_r$ from successive blocks of sizes
$k_1,\ldots,k_r$. This gives a uniformly distributed
ordered family of disjoint classes of the prescribed sizes.
For any fixed injection $\phi:[m]\to V(G_n)$, its images
occupy uniformly distributed distinct positions in the
ordering. The probability that they have the prescribed
class memberships is therefore
$q_N=\prod_{j=1}^r(k_j)_{m_j}/(N)_m$, where
$(a)_b=a(a-1)\cdots(a-b+1)$ and $(a)_0=1$.

Write $C_n=N(F,G_n;V(G_n)^m)$. By linearity of expectation,
the expected labelled restricted count is $q_NC_n$.
Each summand in the symmetrized count also assigns exactly
$m_j$ vertices to $U_j$ for every $j$, so its expectation
is likewise $q_NC_n$. Averaging the assumed counting
identity thus gives, under either hypothesis,
\[
\frac{\prod_{j=1}^r(k_j)_{m_j}}{(N)_m}\,C_n
=\prod_{j=1}^r k_j^{m_j}+o(N^m).
\]
The error remains $o(N^m)$ because it is uniform over all
admissible class families. Since $k_j=\alpha_jN+O(1)$
and every $\alpha_j>0$, we have
$\prod_j(k_j)_{m_j}=\Theta(N^m)$ and
$\prod_jk_j^{m_j}/\prod_j(k_j)_{m_j}\to1$.
Consequently,
\[
\frac{C_n}{(N)_m}
=\frac{\prod_jk_j^{m_j}}{\prod_j(k_j)_{m_j}}
+\frac{o(N^m)}{\prod_j(k_j)_{m_j}}
=1+o(1).
\]

Fix an edge $xy\in E(F)$. Every copy counted by $C_n$
maps $(x,y)$ to an ordered adjacent pair in $G_n$,
giving $2e(G_n)$ choices. Once these images are fixed,
there are at most $(N-2)_{m-2}$ injective assignments
of the remaining vertices. Hence
$C_n\le2e(G_n)(N-2)_{m-2}$, and therefore
\[
1-o(1)
\le\frac{C_n}{(N)_m}
\le\frac{2e(G_n)}{N(N-1)}
\le1.
\]
It follows that the number of missing edges,
$M_n=\binom N2-e(G_n)$, is $o(N^2)$.

For every $U\subseteq V(G_n)$, the number of missing
edges within $U$ is at most $M_n$. Thus
$\bigl|e(G_n[U])-|U|^2/2\bigr|
\le M_n+N/2=o(N^2)$ uniformly over all $U$.
This proves that $(G_n)$ is $1$-quasirandom.
\end{proof}

\subsection{Walsh decompositions and restricted averages}

For $S=\{i_1<\cdots<i_k\}\subseteq[m]$, write
$x_S=(x_{i_1},\ldots,x_{i_k})$. Every real-valued
$f\in L^2([0,1]^m)$ has a unique orthogonal
\emph{generalized Walsh decomposition}
$f=\sum_{S\subseteq[m]}f_S$, where $f_\varnothing=\int f$
is constant and, for nonempty $S$, the component $f_S$
depends only on $x_S$ and has integral zero in each
coordinate in $S$. Explicitly,
\[
f_S=\sum_{T\subseteq S}(-1)^{|S|-|T|}P_Tf,
\qquad
P_Tf(x_T)
=\int_{[0,1]^{[m]\setminus T}}
f(x)\,dx_{[m]\setminus T},
\]
where $P_Tf$ is regarded as a function on $[0,1]^m$
independent of the coordinates outside $T$; see
\cite[Section~2]{HHL}. A function of $k$ variables is
\emph{alternating} if interchanging any two coordinates
changes its sign.

For symmetric functions, we collect the Walsh components
of the same order. For $1\le k\le m$, let $\HH_{m,k}$
be the subspace of $L^2([0,1]^m;\RR)$ consisting of
functions of the form
\[
\sum_{\substack{S\subseteq[m]\\|S|=k}}g(x_S),
\]
where $g\in L^2([0,1]^k;\RR)$ is symmetric and has
integral zero in each coordinate. Let $\HH_{m,0}$ be
the space of constants. These spaces give the orthogonal
decomposition of the symmetric subspace by Walsh order.
In particular, the first level is the space of mean-zero
additive functions:
\[
\HH_{m,1}
=\left\{\sum_{i=1}^m h(x_i):
h\in L^2([0,1];\RR),\ \int_0^1h=0\right\}.
\]

Given an admissible profile $(\bm,\balpha)$, let
$\cM_{\bm,\balpha}$ be the space of symmetric real-valued
functions in $L^2([0,1]^m)$ whose integrals vanish over
every product $A_1^{m_1}\times\cdots\times A_r^{m_r}$,
where $A_1,\ldots,A_r\subseteq[0,1]$ are pairwise
disjoint measurable sets with $\lambda(A_j)=\alpha_j$.
Janson's characterization describes this nullspace as an
orthogonal sum of Walsh levels and determines the levels
by a polynomial divisibility criterion.

To state this criterion, let
$e_k(y_1,\ldots,y_m)
=\sum_{1\le i_1<\cdots<i_k\le m}y_{i_1}\cdots y_{i_k}$
be the elementary symmetric polynomial of degree $k$,
with $e_0=1$, and define
\begin{equation}\label{eq:grouped-polynomial}
E_k(z_1,\ldots,z_r)
=e_k(\underbrace{z_1,\ldots,z_1}_{m_1},\ldots,
     \underbrace{z_r,\ldots,z_r}_{m_r}),
\qquad 0\le k\le m.
\end{equation}
Put $L_{\balpha}(z)=\sum_{j=1}^r\alpha_jz_j$, with
divisibility understood in $\RR[z_1,\ldots,z_r]$.
If $\sum_j\alpha_j=1$, define the set of
\emph{bad indices} by
$\cK=\{k\in[m]:L_{\balpha}\mid E_k\}$;
if $\sum_j\alpha_j<1$, set $\cK=\varnothing$.

\begin{theorem}[Janson
{\cite[Theorems~4.4 and~5.5, (5.3)]{J26}}]
\label{thm:nullspace}
For every admissible profile $(\bm,\balpha)$,
\[
\cM_{\bm,\balpha}
=\bigoplus_{k\in\cK}\HH_{m,k},
\]
where the sum is orthogonal and the empty sum denotes
$\{0\}$.
\end{theorem}

Thus the bad indices are precisely the Walsh orders
annihilated by all the prescribed restricted averages.
The index zero is never bad, since a nonzero constant
has nonzero integral over every admissible product.
Applying Theorem~\ref{thm:nullspace} to the \rev{centred}
symmetrized subgraph kernel shows that the graphon
integral identities in Lemma~\ref{lem:transfer} are
equivalent to
\begin{equation}\label{eq:kernel-test}
\tPsi_{F,W}-p^{e(F)}
\in\bigoplus_{k\in\cK}\HH_{m,k}.
\end{equation}
Together with Lemma~\ref{lem:transfer}, this reduces
the forcing problem to determining whether
\eqref{eq:kernel-test} implies $W=p$ almost everywhere.
In particular, if $\cK\subseteq\{1\}$, the only possible
discrepancy from the constant kernel is additive.

For labelled counts, the subgraph kernel need not be
symmetric, so Theorem~\ref{thm:nullspace} does not apply
directly. We instead use the following consequence of
the functional characterization of Hatami, Hatami and Li.

\begin{lemma}[Hatami--Hatami--Li
{\rev{\cite[Theorem~3.1 and Corollary~3.4]{HHL}}}]
\label{lem:HHL}
Let $\alpha_1,\ldots,\alpha_m>0$ satisfy
$\sum_{i=1}^m\alpha_i\le1$, and let
$f\in L^2([0,1]^m;\RR)$ satisfy
$\int_{A_1\times\cdots\times A_m}f=0$ for every
pairwise disjoint family of measurable sets
$A_1,\ldots,A_m\subseteq[0,1]$ with
$\lambda(A_i)=\alpha_i$.
If $\sum_i\alpha_i<1$, then $f=0$ almost everywhere.
If $\sum_i\alpha_i=1$, the generalized Walsh
decomposition $f=\sum_{S\subseteq[m]}f_S$ satisfies
$f_\varnothing=0$, and every component $f_S$ with
$|S|\ge2$ is alternating in the coordinates indexed by $S$.
\end{lemma}

We shall use only the necessary conditions in Lemma~\ref{lem:HHL}.
These imply that full symmetrization leaves at most a mean-zero additive term.
Indeed, for $|S|\ge2$, let $\tau$ interchange two
coordinates in $S$. Alternation gives
$f_S\circ\tau=-f_S$, while
$\Sym(f_S\circ\tau)=\Sym f_S$ by the definition of
symmetrization. Hence $\Sym f_S=0$. For a singleton
component, write $f_{\{i\}}(x)=h_i(x_i)$, where
$\int_0^1h_i=0$. Then
$\Sym f_{\{i\}}=m^{-1}\sum_{j=1}^m h_i(x_j)$.
Consequently, when $\sum_i\alpha_i=1$,
\[
\Sym f(x_1,\ldots,x_m)=\sum_{j=1}^m h(x_j),
\qquad
h=\frac1m\sum_{i=1}^m h_i,
\qquad
\int_0^1h=0.
\]
Thus the labelled counting condition also reduces, after
symmetrization, to an additive discrepancy. This allows
the labelled classification to use the same additive
rigidity statement as the symmetrized classification.

\section{Balanced counts and additive rigidity}\label{sec:balanced}

\begin{revision}
We prove Theorem~\ref{thm:balanced} by completing the
algebraic reduction of Janson and S\'os. Their reduction
associates a nonconstant triple $(U,V,S)$ satisfying a
finite algebraic system with any counterexample to the
balanced-count conjecture. The four equations indexed by
$0,1,m-1,m$ suffice to force $U=V=S$. The key step is the
following scalar inequality for exponents determined by
the degree sequence of $F$.
\end{revision}

\begin{lemma}\label{lem:scalar}
Let $m\ge3$, and let $a_1,\ldots,a_m\in[0,1]$ satisfy
$\sum_{i=1}^m a_i=m-2$, with at least one $a_i\in(0,1)$.
Set $g(t)=\sum_{i=1}^m t^{a_i}-(m-1)t$ for $t>0$.
Then no pair $0<x<1<y$ satisfies $g(x)=y$ and $g(y)=x$.
\end{lemma}

\begin{proof}
Since $g(1)=1$ and
$g'(1)=\sum_i a_i-(m-1)=-1$, the tangent line to $g$
at $1$ is $2-t$. Define
$h(t)=g(t)-(2-t)
=\sum_{i=1}^m t^{a_i}-(m-2)t-2$ for $t>0$.
Then $h$ is smooth on $(0,\infty)$ and
$h(1)=h'(1)=0$. Moreover,
\[
h''(t)=\sum_{i=1}^m a_i(a_i-1)t^{a_i-2}<0,
\qquad
h'''(t)=\sum_{i=1}^m
a_i(a_i-1)(a_i-2)t^{a_i-3}>0
\]
for every $t>0$. Indeed, exponents in $\{0,1\}$
contribute zero to both sums, whereas each exponent
in $(0,1)$ contributes negatively to the first and
positively to the second. Thus $h'$ is strictly
decreasing and $h''$ is strictly increasing.
As $h'(1)=0$, the function $h$ is strictly increasing
on $(0,1)$ and strictly decreasing on $(1,\infty)$.
Consequently, $h(t)<h(1)=0$ whenever $t\ne1$.

Suppose that $0<x<1<y$ satisfies $g(x)=y$ and $g(y)=x$.
Then $h(x)=h(y)=x+y-2<0$, and hence $x+y<2$.
Putting $A=1-x$ and $B=y-1$, we obtain $0<B<A<1$.

Define $D(t)=h(1+t)-h(1-t)$ for $|t|<1$.
Differentiation gives
$D'(t)=h'(1+t)+h'(1-t)$ and
$D''(t)=h''(1+t)-h''(1-t)$.
Thus $D(0)=D'(0)=0$, and the strict increase of $h''$
implies $D''(t)>0$ for $0<t<1$.
Integrating twice yields
$D(t)=\int_0^t(t-s)D''(s)\,ds>0$ for $0<t<1$.
These integrations involve only positive arguments of $h$;
in particular, for $0\le t\le A$, those arguments lie
in $[1-A,1+A]\subset(0,\infty)$.
Taking $t=A$ therefore gives $h(1-A)<h(1+A)$.

Finally, since $A>B>0$ and $h$ is strictly decreasing
on $(1,\infty)$,
\[
h(x)=h(1-A)<h(1+A)<h(1+B)=h(y),
\]
contrary to $h(x)=h(y)$.
\end{proof}

\begin{remark}
The assumption that some $a_i\in(0,1)$ is necessary.
If all $a_i\in\{0,1\}$, then $\sum_i a_i=m-2$ implies
that exactly two exponents are zero and the others are one.
Thus $g(t)=2-t$, and every $x\in(0,1)$, with $y=2-x$,
satisfies $g(x)=y$ and $g(y)=x$.
\end{remark}

For a graph with degree sequence $d_1,\ldots,d_m$ and
$e>1$ edges, the handshaking identity gives
$\sum_{i=1}^m(1-d_i/e)=m-2$.
Applying Lemma~\ref{lem:scalar} with exponents
$a_i=1-d_i/e$ yields the following obstruction to
solutions on opposite sides of $1$.

\begin{proposition}\label{prop:degrees}
Let $F$ be a simple graph with vertex set
$\{v_1,\ldots,v_m\}$ and $e=e(F)>1$ edges, and write
$d_i=d_F(v_i)$ for $i\in[m]$. Then no pair $0<u<1<v$
satisfies
\begin{align}
\sum_{i=1}^m u^{e-d_i}
&=(m-1)u^e+v^e,
\label{eq:degreeu}\\
\sum_{i=1}^m v^{e-d_i}
&=u^e+(m-1)v^e.
\label{eq:degreev}
\end{align}
\end{proposition}

\begin{proof}
Since $F$ is simple and $e>1$, we have $m\ge3$.
Put $a_i=1-d_i/e$ for $i\in[m]$.
The bounds $0\le d_i\le e$ give $a_i\in[0,1]$,
while the handshaking identity yields
$\sum_i a_i=m-e^{-1}\sum_i d_i=m-2$.

To verify that some $a_i\in(0,1)$, choose distinct
edges $e_1,e_2\in E(F)$. These are distinct two-element
vertex sets, so there is a vertex
$v_\ell\in e_1\setminus e_2$.
Since $e_1$ is incident with $v_\ell$ and $e_2$ is not,
we have $1\le d_\ell\le e-1$, and hence
$0<a_\ell<1$. Thus $a_1,\ldots,a_m$ satisfy all the
hypotheses of Lemma~\ref{lem:scalar}.

Suppose that $0<u<1<v$ satisfies
\eqref{eq:degreeu} and \eqref{eq:degreev}.
Set $x=u^e$ and $y=v^e$. Since $e>0$,
we have $0<x<1<y$. Moreover, $u,v>0$ implies
$x^{a_i}=u^{e-d_i}$ and $y^{a_i}=v^{e-d_i}$
for every $i\in[m]$. The two equations therefore become
$\sum_i x^{a_i}=(m-1)x+y$ and
$\sum_i y^{a_i}=x+(m-1)y$.
For $g(t)=\sum_i t^{a_i}-(m-1)t$, these identities
give $g(x)=y$ and $g(y)=x$, contradicting
Lemma~\ref{lem:scalar}.
\end{proof}

Proposition~\ref{prop:degrees} establishes the degree-sequence
assertion proposed in \cite[Remark~7.8]{JS}.
Isolated vertices contribute exponents equal to one and
require no separate treatment. We now apply this obstruction
to the algebraic system arising from balanced counts.
The equations indexed by $0$ and $m$ determine the affine
parameters $c$ and $b$, while those indexed by $1$ and
$m-1$ lead, after rescaling, to the degree equations.
The following theorem shows that these four equations
already force the constant solution. In particular, it
resolves \cite[Problem~6.5]{J26}, which asks to exclude
nonconstant solutions of the full system.

\begin{theorem}\label{thm:algebra}
Let $F$ be a simple graph on $m$ vertices with
$e=e(F)>1$. Suppose that $U,V,S\ge0$ and
$c,b\in\RR$, with $(c,b)\ne(0,0)$, satisfy
\begin{equation}\label{eq:system}
\sum_{\substack{T\subseteq V(F)\\|T|=k}}
U^{e_F(T)}V^{e_F(T^c)}S^{e_F(T,T^c)}
=\binom{m}{k}(c+bk)
\end{equation}
for every $k\in\{0,1,m-1,m\}$.
Then $U=V=S>0$ and $b=0$.
\end{theorem}

\begin{proof}
Enumerate the vertices of $F$ as $v_1,\ldots,v_m$ and
write $d_i=d_F(v_i)$. The equations with $k=0$ and
$k=m$ give
\begin{equation}\label{eq:ends}
V^e=c,\qquad U^e=c+mb.
\end{equation}
For $T=\{v_i\}$, we have $e_F(T)=0$,
$e_F(T^c)=e-d_i$ and $e_F(T,T^c)=d_i$.
Taking $T=\{v_i\}$ and $T=V(F)\setminus\{v_i\}$,
respectively, therefore gives
\begin{align}
\sum_{i=1}^m V^{e-d_i}S^{d_i}
&=m(c+b),
\label{eq:one}\\
\sum_{i=1}^m U^{e-d_i}S^{d_i}
&=m\bigl(c+(m-1)b\bigr).
\label{eq:almostall}
\end{align}

First suppose that $U=V$. Equation~\eqref{eq:ends}
gives $b=0$ and $c=U^e$. The assumption
$(c,b)\ne(0,0)$ then implies $U>0$.
Equation~\eqref{eq:one} becomes
$\sum_i U^{e-d_i}S^{d_i}=mU^e$.
The left-hand side is strictly increasing as a function
of $S\ge0$: all coefficients $U^{e-d_i}$ are positive,
all terms are nondecreasing, and at least one $d_i$
is positive. Its value at $S=U$ is $mU^e$, so
$S=U=V$.

Suppose now that $U\ne V$. Replacing $T$ by $T^c$
in~\eqref{eq:system} interchanges $U$ and $V$ and
replaces $k$ by $m-k$. Since
$\binom{m}{m-k}=\binom{m}{k}$ and
$c+b(m-k)=(c+mb)-bk$, the parameters
$(V,U,S,c+mb,-b)$ also satisfy the four equations.
Here the index set $\{0,1,m-1,m\}$ is invariant under
$k\mapsto m-k$, and the transformation
$(c,b)\mapsto(c+mb,-b)$ is invertible, so it preserves
the condition $(c,b)\ne(0,0)$.
We may therefore assume that $U>V\ge0$.
By~\eqref{eq:ends}, $mb=U^e-V^e>0$, and hence $b>0$.

We first show that $S<U$. If $S\ge U$, then
$U^{e-d_i}S^{d_i}\ge U^e$ for every $i$.
Equation~\eqref{eq:almostall} consequently yields
\[
mU^e
\le\sum_{i=1}^m U^{e-d_i}S^{d_i}
=m\bigl(c+(m-1)b\bigr)
=m(U^e-b)<mU^e,
\]
a contradiction.

Next suppose that $V=0$. Then $c=0$ and
$U^e=mb>0$. Let $q$ be the number of vertices of
degree $e$. In~\eqref{eq:one}, a summand is zero
when $d_i<e$ and equals $S^e$ when $d_i=e$.
Thus $qS^e=mb=U^e$.
There is at most one vertex of degree $e$: if two
distinct vertices were incident with every edge,
every edge would join those two vertices, and
simplicity would imply $e\le1$.
Hence $q\le1$, and the inequality $S<U$ gives
$U^e=qS^e\le S^e<U^e$, a contradiction.
Therefore $V>0$.

If $S\le V$, then
$\sum_i V^{e-d_i}S^{d_i}\le mV^e<m(V^e+b)=m(c+b)$,
contrary to~\eqref{eq:one}. We have therefore obtained
$0<V<S<U$.

Finally, substitute $c=V^e$ and
$b=(U^e-V^e)/m$ into~\eqref{eq:one} and
\eqref{eq:almostall}. Dividing by $S^e>0$ and setting
$u=V/S$ and $v=U/S$ gives
\[
\sum_{i=1}^m u^{e-d_i}=(m-1)u^e+v^e,
\qquad
\sum_{i=1}^m v^{e-d_i}=u^e+(m-1)v^e.
\]
Since $0<u<1<v$, these equations contradict
Proposition~\ref{prop:degrees}.
Thus $U\ne V$ is impossible, and the first case
gives $U=V=S>0$ and $b=0$.
\end{proof}

\begin{remark}
The condition $(c,b)\ne(0,0)$ is necessary.
For $F=K_3$, take $U=V=0$, $S=1$ and $c=b=0$.
For every $T\subseteq V(F)$, at least one of
$e_F(T)$ and $e_F(T^c)$ is positive, so every
monomial on the left-hand side of~\eqref{eq:system}
vanishes. Thus the full system holds for
$k=0,\ldots,m$, although $U,V,S$ are not all equal.
\end{remark}

Although the full system contains an equation for every
$k=0,\ldots,m$, Theorem~\ref{thm:algebra} uses only the
four indices $0,1,m-1,m$ and requires no connectedness
or regularity assumption on $F$. The following implication
of the Janson--S\'os reduction connects this algebraic
rigidity with balanced subgraph counts.

\begin{lemma}[Janson--S\'os
{\cite[Lemma~6.3]{JS}}]\label{lem:JS}
Let $F$ be a graph on $m$ vertices with $e(F)>0$.
If $\tP^p(F;1/m,\ldots,1/m)$ is not a
$p$-quasirandom property for some $p\in(0,1]$, then
there exist $U,V,S\ge0$, not all equal, and
$c,b\in\RR$, with $(c,b)\ne(0,0)$, satisfying
\eqref{eq:system} for every $k=0,\ldots,m$.
\end{lemma}

The triple $(U,V,S)$ in Lemma~\ref{lem:JS} arises from
a two-type graphon: $U,V$ are its within-type values,
and $S$ is its between-type value. The reduction ensures
that the symmetrized subgraph kernel is additive, so
its value on a tuple with $k$ coordinates of the first
type is affine in $k$. Expressing this value as an
average over all placements of those coordinates gives
\eqref{eq:system}. \rev{Thus a counterexample to the balanced-count
conjecture would produce the nonconstant solution excluded by
Theorem~\ref{thm:algebra}, yielding the following proof.}

\begin{proof}[Proof of Theorem~\ref{thm:balanced}]
By Lemma~\ref{lem:uniform-counting}, every
$p$-quasirandom graph sequence satisfies both the
labelled and symmetrized balanced counting conditions.
It remains to prove the converse implications.
The case $p=1$ follows from Lemma~\ref{lem:p-one},
so assume that $0<p<1$.

Suppose that $\tP^p(F;1/m,\ldots,1/m)$ is not a
$p$-quasirandom property. By Lemma~\ref{lem:JS},
there exist $U,V,S\ge0$, not all equal, and
$c,b\in\RR$, with $(c,b)\ne(0,0)$, satisfying
\eqref{eq:system} for every $k=0,\ldots,m$.
In particular, the equations indexed by $0,1,m-1,m$
hold. Since $F$ is simple and $e(F)>1$,
Theorem~\ref{thm:algebra} gives $U=V=S>0$ and $b=0$,
contradicting the choice of the triple. Thus the
symmetrized balanced counting condition characterizes
$p$-quasirandomness.

Finally, suppose that a graph sequence satisfies the
labelled balanced counting condition. Since all classes
have the same prescribed size, every permutation of an
admissible tuple is admissible. Averaging the labelled
identities over $\mathfrak S_m$, with the uniform error
unchanged, yields the symmetrized counting condition.
The sequence is therefore $p$-quasirandom, proving
the labelled assertion.
\end{proof}

The two classifications use the following consequence
of Theorem~\ref{thm:balanced}, which is independent of
the class profile. An additive symmetrized subgraph
kernel has the same integral over every equipartition
product. If its mean is positive, these integrals
satisfy the balanced counting condition at a positive
density, forcing the underlying graphon to be constant.

\begin{corollary}\label{cor:additive}
Let $F$ be a simple graph on $m$ vertices with $e(F)>1$,
let $W$ be a graphon, and let $h\in L^1([0,1];\RR)$
satisfy $\mu=\int_0^1h(t)\,dt>0$. If
\[
\tPsi_{F,W}(x_1,\ldots,x_m)=\sum_{i=1}^m h(x_i)
\]
for almost every $(x_1,\ldots,x_m)\in[0,1]^m$, then
$W=q$ almost everywhere on $[0,1]^2$ and
$h=q^{e(F)}/m$ almost everywhere on $[0,1]$, where
$q=(m\mu)^{1/e(F)}\in(0,1]$.
\end{corollary}

\begin{proof}
Put $e=e(F)$. Each function $h(x_i)$ is integrable
on $[0,1]^m$, so integrating the assumed identity
and applying Fubini's theorem gives
$0<m\mu=\int_{[0,1]^m}\tPsi_{F,W}\le1$,
where the upper bound follows from
$0\le\tPsi_{F,W}\le1$. Thus
$q=(m\mu)^{1/e}$ is well defined and belongs to $(0,1]$.

Let $A_1,\ldots,A_m$ be any measurable partition
of $[0,1]$ with $\lambda(A_i)=1/m$ for every $i$.
By the assumed identity and Fubini's theorem,
\[
\begin{aligned}
\int_{A_1\times\cdots\times A_m}\tPsi_{F,W}
&=\sum_{i=1}^m
  \left(\prod_{j\ne i}\lambda(A_j)\right)
  \int_{A_i}h(t)\,dt\\
&=\frac{\mu}{m^{m-1}}
 =q^e\prod_{i=1}^m\lambda(A_i).
\end{aligned}
\]
Hence $W$ satisfies the symmetrized balanced graphon
counting condition with density parameter $q$.
Since $0<q\le1$, Theorem~\ref{thm:balanced} and
the graph-sequence--graphon correspondence
\cite[Lemma~3.4]{JS} imply that $W=q$ almost everywhere.

For each pair $i<j$, Fubini's theorem shows that
the set of tuples for which $W(x_i,x_j)\ne q$
has measure zero. Outside the finite union of these
exceptional sets, every factor in every product
defining $\tPsi_{F,W}$ equals $q$.
Thus $\tPsi_{F,W}=q^e$ almost everywhere, and the
assumed identity gives $\sum_i h(x_i)=q^e$ almost
everywhere. Integrating with respect to
$x_2,\ldots,x_m$ yields
$h(x_1)+(m-1)\mu=q^e=m\mu$ for almost every $x_1$.
Therefore $h=\mu=q^e/m$ almost everywhere.
\end{proof}

By \eqref{eq:kernel-test}, the inclusion
$\cM_{\bm,\balpha}\subseteq\HH_{m,1}$ ensures that
the discrepancy from the constant subgraph kernel is
additive. Corollary~\ref{cor:additive} therefore yields
the following forcing criterion, with no restriction
on the number of classes.

\begin{corollary}\label{cor:additive-profile}
Let $F$ be a simple graph on $m$ vertices with $e(F)>1$,
let $0<p\le1$, and let $(\bm,\balpha)$ be an admissible
profile. If $\cM_{\bm,\balpha}\subseteq\HH_{m,1}$,
then $\tP^p_{\bm}(F;\balpha)$ is a
$p$-quasirandom property.
\end{corollary}

\begin{proof}
Every $p$-quasirandom graph sequence satisfies the
prescribed counting condition by
Lemma~\ref{lem:uniform-counting}. Conversely, the case
$p=1$ follows from Lemma~\ref{lem:p-one}, so assume
that $0<p<1$.

Put $e=e(F)$. By Lemma~\ref{lem:transfer}, it suffices
to consider a graphon $W$ satisfying the prescribed
symmetrized integral identities. By
\eqref{eq:kernel-test},
$\tPsi_{F,W}-p^e\in\cM_{\bm,\balpha}\subseteq\HH_{m,1}$.
Thus there exists $h_0\in L^2([0,1];\RR)$ with
$\int_0^1h_0=0$ such that
$\tPsi_{F,W}-p^e=\sum_{i=1}^m h_0(x_i)$
almost everywhere. Set $h=h_0+p^e/m$. Then
$h\in L^1([0,1];\RR)$, $\int_0^1h=p^e/m>0$, and
$\tPsi_{F,W}=\sum_{i=1}^m h(x_i)$ almost everywhere.
Corollary~\ref{cor:additive} therefore gives $W=q$
almost everywhere, where
$q=(m\int_0^1h)^{1/e}=p$.
This proves the converse.
\end{proof}

Corollary~\ref{cor:additive-profile} settles
\cite[Conjecture~6.3]{J26}, while
Theorem~\ref{thm:algebra} resolves
\cite[Problem~6.5]{J26}. The hypothesis $e(F)>1$
is essential to the rigidity argument: when $F$
has exactly one edge, the exponents in
Lemma~\ref{lem:scalar} are two zeros and $m-2$ ones,
so its strictness hypothesis fails. The next section
completes the repeated-class classification. We first
show that, for profiles with at least three classes,
only additive obstructions can occur. Additive rigidity
then settles the case $e(F)>1$, leaving the single-edge
case to determine the exceptions.

\section{Symmetrized repeated-class counts}\label{sec:repeated}

In this section, we show that profiles with at least
three classes admit only additive obstructions and
then complete the classification.
By Theorem~\ref{thm:nullspace}, excluding higher-order
bad indices amounts to showing that $E_k$, for $k\ge2$,
has no linear divisor with all coefficients positive
when $r\ge3$. To prove this non-divisibility statement,
we first consider a two-value specialization of the
elementary symmetric polynomial. The following lemma
provides the root simplicity needed in the argument.

\begin{lemma}\label{lem:app520-simple}
Let $A,B$ be positive integers, put $N=A+B$, and let
$1\le k\le N$. Then every negative real zero of
\[
R_{A,B,k}(t)
=e_k(\underbrace{1,\ldots,1}_{A},
     \underbrace{t,\ldots,t}_{B})
\]
is simple.
\end{lemma}

\begin{proof}
If $k=N$, then $R_{A,B,k}(t)=t^B$, so there are no
negative real zeros. Assume henceforth that $1\le k<N$,
and put $d=N-k$ and $H(x)=x^A(x-1)^B$.
The Leibniz formula and the homogeneity of $e_k$ give,
for $x\ne0$,
\begin{equation}\label{eq:derivative-R}
H^{(d)}(x)
=d!\,e_k(\underbrace{x,\ldots,x}_{A},
         \underbrace{x-1,\ldots,x-1}_{B})
=d!\,x^kR_{A,B,k}\left(1-\frac1x\right).
\end{equation}

We first show by induction that every $H^{(j)}$,
for $0\le j\le N-1$, has all its zeros in $[0,1]$,
with every zero in $(0,1)$ simple.
The assertion holds for $H$, whose only zeros are
$0$ and $1$.

Suppose it holds for $P=H^{(j)}$, where $j\le N-2$,
and put $n=\deg P=N-j\ge2$.
Let $z_1<\cdots<z_\ell$ be the distinct zeros of $P$,
with multiplicities $r_1,\ldots,r_\ell$.
Each $z_i$ lies in $[0,1]$, and $\sum_i r_i=n$.
If $r_i\ge2$, then $z_i$ is a zero of $P'$ of
multiplicity exactly $r_i-1$. Rolle's theorem also
gives a zero of $P'$ in each interval $(z_i,z_{i+1})$.
These gap zeros are distinct from the retained zeros,
so they account for at least
$\sum_i(r_i-1)+(\ell-1)=n-1$ zeros of $P'$, counted
with multiplicity. Since $\deg P'=n-1$, they account
for all its zeros. Thus each gap contains exactly
one zero, which is simple. Moreover, retained zeros
can occur only at the endpoints, since every interior
zero of $P$ is simple by the induction hypothesis.
Hence all zeros of $P'$ lie in $[0,1]$, and every
interior zero is simple. This completes the induction.

Let $t_0<0$ be a zero of $R_{A,B,k}$, and set
$x_0=1/(1-t_0)$. Then $x_0\in(0,1)$ and
$t_0=1-1/x_0$, so~\eqref{eq:derivative-R} gives
$H^{(d)}(x_0)=0$. Since $1\le d\le N-1$, this
zero is simple, and therefore $H^{(d+1)}(x_0)\ne0$.
Differentiating~\eqref{eq:derivative-R} and using
$R_{A,B,k}(t_0)=0$ yields
\[
H^{(d+1)}(x_0)
=d!\,x_0^{k-2}R'_{A,B,k}(t_0).
\]
As $x_0>0$, the prefactor is nonzero.
Thus $R'_{A,B,k}(t_0)\ne0$, proving the assertion.
\end{proof}

With the simplicity of the negative roots established,
we now turn to the grouped polynomials $E_k$.
A linear divisor with positive coefficients would first
force those coefficients to be proportional to the group
multiplicities. A variation of two group variables within
its zero hyperplane then rules out such a divisor when
$r\ge3$.

\begin{theorem}\label{thm:app520-factor}
Let $r\ge3$, let $m_1,\ldots,m_r$ be positive integers
with $\sum_{i=1}^r m_i=m$, and let $2\le k\le m$.
Then $E_k$ has no linear divisor
$L(z)=\sum_{i=1}^r\alpha_i z_i$ with
$\alpha_i>0$ for every $i\in[r]$.
\end{theorem}

\begin{proof}
Suppose, to the contrary, that $E_k=LQ$ for some
$Q\in\RR[z_1,\ldots,z_r]$. If $k=m$, then
$E_m(z)=\prod_{i=1}^r z_i^{m_i}$. At the point defined by
$z_i=1$ for $i\ne r$ and
$z_r=-\sum_{i\ne r}\alpha_i/\alpha_r$, all coordinates
are nonzero and $L(z)=0$, whereas $E_m(z)\ne0$.
This contradicts divisibility. Hence $2\le k<m$.

To distinguish derivatives in the group variables from
those in the individual variables of $e_k$, partition
$[m]$ into sets $I_1,\ldots,I_r$ with $|I_i|=m_i$.
Writing $y=(y_1,\ldots,y_m)$, we obtain $E_k(z)$ from
$e_k(y)$ by setting $y_s=z_i$ whenever $s\in I_i$.
Thus differentiation with respect to $z_i$ sums the
partial derivatives with respect to the $m_i$ variables
indexed by $I_i$.

Fix $j\in[r]$, and put $A=m-m_j$ and $B=m_j$.
Consider the point $z^*$ given by
$z_i^*=1$ for $i\ne j$ and
\[
z_j^*=v=-\frac{\sum_{i\ne j}\alpha_i}{\alpha_j}<0.
\]
Since $L(z^*)=0$, divisibility gives
$R_{A,B,k}(v)=E_k(z^*)=0$. By
Lemma~\ref{lem:app520-simple},
$R_{A,B,k}'(v)\ne0$. Differentiating $E_k=LQ$ at $z^*$
yields
\[
\partial_iE_k(z^*)=\alpha_iQ(z^*)
\qquad(i\in[r]).
\]
In particular,
$\alpha_jQ(z^*)=\partial_jE_k(z^*)
=R_{A,B,k}'(v)\ne0$, so $Q(z^*)\ne0$.

All individual variables outside $I_j$ have value $1$
at this point. By symmetry of $e_k$, its partial
derivatives with respect to these variables have a
common value, say $\rho$. The chain rule therefore gives
$\partial_iE_k(z^*)=m_i\rho$ for every $i\ne j$.
Consequently,
$\alpha_i/m_i=\rho/Q(z^*)$ is independent of $i\ne j$.
Given any two distinct indices $a,b\in[r]$, the
assumption $r\ge3$ allows a choice of
$j\notin\{a,b\}$. Hence
$\alpha_a/m_a=\alpha_b/m_b$, and all coefficients of
$L$ are proportional to the corresponding multiplicities.
Absorbing the common positive factor into $Q$, we may
therefore assume that $L(z)=\sum_{i=1}^r m_i z_i$.

It remains to exclude this divisor. Fix $j\in[r]$
again, and put $A=m-m_j$, $B=m_j$, $u=1$, and $v=-A/B$.
At the point $z^*$ with $z_i^*=u$ for $i\ne j$ and
$z_j^*=v$, we have $L(z^*)=Au+Bv=0$. Define
\[
R=e_{k-1}(u^{[A-1]},v^{[B]}),
\qquad
S=e_{k-1}(u^{[A]},v^{[B-1]}).
\]
Here and below, these expressions denote values at the
fixed pair $(u,v)=(1,-A/B)$. For $i\ne j$, the chain
rule gives $\partial_iE_k(z^*)=m_iR$, whereas
$\partial_jE_k(z^*)=BS$. On the other hand,
differentiating $E_k=LQ$ at $z^*$ gives
$\partial_iE_k(z^*)=m_iQ(z^*)$ for every $i$.
Thus $R=S=Q(z^*)$. Moreover,
$R_{A,B,k}(v)=E_k(z^*)=0$, and
Lemma~\ref{lem:app520-simple} gives
$BS=R_{A,B,k}'(v)\ne0$. Therefore $R=S\ne0$.

Since $A=\sum_{i\ne j}m_i\ge r-1\ge2$, the quantities
\[
D=e_{k-2}(u^{[A-2]},v^{[B]}),
\qquad
E=e_{k-2}(u^{[A-1]},v^{[B-1]})
\]
are well defined. Expanding $R$ and $S$ by their
respective final entries gives
\begin{align*}
R&=e_{k-1}(u^{[A-1]},v^{[B-1]})+vE,\\
S&=e_{k-1}(u^{[A-1]},v^{[B-1]})+uE.
\end{align*}
Hence $0=R-S=(v-u)E$. Since $v\ne u$, it follows
that $E=0$. If $k=2$, this contradicts $E=e_0=1$.
We may therefore assume that $k\ge3$.

Consider the homogeneous polynomial
$J(s,t)=e_{k-1}(s^{[A-1]},t^{[B]})$, of degree $k-1$.
At $(u,v)$, its partial derivatives are
$\partial_sJ(u,v)=(A-1)D$ and
$\partial_tJ(u,v)=BE$. Euler's identity therefore gives
\[
u(A-1)D+vBE=(k-1)R.
\]
Since $E=0$, $u=1$, $A\ge2$, and $R\ne0$, we conclude
that $D\ne0$.

Choose distinct indices $a,b\in[r]\setminus\{j\}$.
Starting from $z^*$, vary the group variables along
$z_a(\varepsilon)=u+\varepsilon/m_a$ and
$z_b(\varepsilon)=u-\varepsilon/m_b$, leaving all
other variables fixed. Then
$L(z(\varepsilon))=L(z^*)+\varepsilon-\varepsilon=0$
for every $\varepsilon\in\RR$. Divisibility implies
that $E_k(z(\varepsilon))$ is identically zero.

Let $y_s(\varepsilon)=z_i(\varepsilon)$ for $s\in I_i$,
and write $h_s=y_s'(0)$. These paths are affine, and
$h_s=1/m_a$ on $I_a$, $h_s=-1/m_b$ on $I_b$, and
$h_s=0$ elsewhere. In particular,
$\sum_{s=1}^m h_s=0$ and
$\sum_{s=1}^m h_s^2=1/m_a+1/m_b$.
The polynomial $e_k$ is affine in each individual
variable, so its diagonal second partial derivatives
vanish. For distinct indices $s,t$ with $h_sh_t\ne0$,
both corresponding variables have value $u$ at
$\varepsilon=0$. Removing these two variables therefore
gives
$\partial_{y_s}\partial_{y_t}e_k(y(0))=D$.
Since all second derivatives of the paths vanish,
the chain rule yields
\begin{align*}
\left.\frac{d^2}{d\varepsilon^2}
E_k(z(\varepsilon))\right|_{\varepsilon=0}
&=D\sum_{\substack{s,t\in[m]\\s\ne t}}h_sh_t\\
&=D\left[
 \left(\sum_{s=1}^m h_s\right)^2
 -\sum_{s=1}^m h_s^2
 \right]\\
&=-D\left(\frac1{m_a}+\frac1{m_b}\right)\ne0.
\end{align*}
This contradicts the identity
$E_k(z(\varepsilon))=0$ for every $\varepsilon$,
and completes the proof.
\end{proof}

\begin{remark}
Both restrictions $r\ge3$ and $k\ge2$ are necessary.
For $k=1$, the polynomial $E_1(z)=\sum_{i=1}^r m_i z_i$
is itself a linear form with positive coefficients.
For two groups with $m_1=m_2=2$, one has
$E_3(z_1,z_2)=2z_1z_2(z_1+z_2)$, which has the
positive-coefficient linear divisor $z_1+z_2$.
\end{remark}

Combining Theorem~\ref{thm:app520-factor} with the
divisibility criterion in Theorem~\ref{thm:nullspace}
gives a complete description of the bad indices for
$r\ge3$. No index $k\ge2$ is bad, as asserted in
\cite[Conjecture~5.20]{J26}, while the first index is
bad precisely for balanced profiles.

\begin{corollary}\label{cor:app520-bad}
Let $(\bm,\balpha)$ be an admissible profile with $r\ge3$.
Then
\[
\cK=
\begin{cases}
\{1\},&\balpha=\bm/m,\\
\varnothing,&\balpha\ne\bm/m.
\end{cases}
\]
\end{corollary}

\begin{proof}
If $\sum_{i=1}^r\alpha_i<1$, then $\cK=\varnothing$
by definition. In this case, $\balpha\ne\bm/m$, since
the entries of $\bm/m$ sum to $1$.

Suppose that $\sum_{i=1}^r\alpha_i=1$.
The divisibility criterion in Theorem~\ref{thm:nullspace},
together with Theorem~\ref{thm:app520-factor}, gives
$\cK\subseteq\{1\}$. Moreover, $1\in\cK$ if and only if
$L_{\balpha}$ divides $E_1(z)=\sum_{i=1}^r m_i z_i$.
Since both polynomials are nonzero and linear, this is
equivalent to $E_1=cL_{\balpha}$ for some constant $c>0$,
or $m_i=c\alpha_i$ for every $i$. Summing these identities
gives $c=m$. Conversely, if $\balpha=\bm/m$, then
$E_1=mL_{\balpha}$. Hence $1\in\cK$ precisely when
$\balpha=\bm/m$, proving the assertion.
\end{proof}

\begin{revision}
To complete Theorem~\ref{thm:repeated}, we combine
Corollary~\ref{cor:app520-bad} with additive rigidity and
the following single-edge criterion of Janson.
\end{revision}

\begin{lemma}[Janson~{\cite[Example~6.10]{J26}}]
\label{lem:single-edge}
Let $F$ be a simple graph with $e(F)=1$, let
$(\bm,\balpha)$ be an admissible profile, and let $0<p<1$.
Then $\tP^p_{\bm}(F;\balpha)$ is a $p$-quasirandom
property if and only if $\cK\cap\{1,2\}=\varnothing$.
\end{lemma}

\begin{proof}[Proof of Theorem~\ref{thm:repeated}]
By Theorem~\ref{thm:nullspace} and
Corollary~\ref{cor:app520-bad},
$\cM_{\bm,\balpha}\subseteq\HH_{m,1}$.
If $e(F)>1$, Corollary~\ref{cor:additive-profile}
therefore shows that
$\tP^p_{\bm}(F;\balpha)$ is a $p$-quasirandom property.

If $e(F)=1$, Lemma~\ref{lem:single-edge} shows that
the property is $p$-quasirandom if and only if
$\cK\cap\{1,2\}=\varnothing$. By
Corollary~\ref{cor:app520-bad}, this is equivalent to
$\balpha\ne\bm/m$. The two cases give the asserted
classification.
\end{proof}

\begin{revision}
The exceptional balanced single-edge profiles admit an
explicit two-type witness. Fix $0<p<1$, let
$\sigma=\mathbf{1}_{[0,1/2)}-\mathbf{1}_{[1/2,1]}$,
choose $0<\varepsilon<\frac12\min\{p,1-p\}$, and set
\[
W(x,y)=p+\varepsilon\bigl(\sigma(x)+\sigma(y)\bigr).
\]
Then $W$ is a nonconstant graphon. If $e(F)=1$, averaging
the positions of its two endpoints gives
\[
\tPsi_{F,W}-p
=\frac{2\varepsilon}{m}\sum_{i=1}^m\sigma(x_i).
\]
Consequently, for every admissible family $A_1,\ldots,A_r$,
\[
\begin{aligned}
&\int_{A_1^{m_1}\times\cdots\times A_r^{m_r}}
  (\tPsi_{F,W}-p)\\
&\qquad=\frac{2\varepsilon}{m}
  \left(\prod_{j=1}^r\alpha_j^{m_j}\right)
  \sum_{j=1}^r\frac{m_j}{\alpha_j}\int_{A_j}\sigma.
\end{aligned}
\]
If $\alpha_j=m_j/m$ for every $j$, the sets form a
partition up to null sets, and the final sum equals
$m\int_0^1\sigma=0$. Lemma~\ref{lem:transfer} therefore
gives a graph sequence satisfying the balanced symmetrized
condition that is not $p$-quasirandom.
\end{revision}

\section{Labelled transversal counts}\label{sec:transversal}

\begin{revision}
We now prove Theorem~\ref{thm:transversal}. The case
$\sum_i\alpha_i<1$ is already covered by
\cite[Theorem~2.11(ii)]{JS}. For complementary proportions,
the calculation following Lemma~\ref{lem:HHL} reduces
the symmetrized discrepancy to an additive function,
so Corollary~\ref{cor:additive} applies when $e(F)>1$.
For graphs with exactly one edge, we use the following
classical edge-count classification; see
\cite[Example~2.10 and Theorem~2.11]{JS}.
\end{revision}

\begin{lemma}[Edge-count classification]
\label{lem:edge-transversal}
Let $0<p<1$, and let $\beta_1,\beta_2>0$ satisfy
$\beta_1+\beta_2\le1$.
Then $\mathcal P^p(K_2;\beta_1,\beta_2)$ is a
$p$-quasirandom property if and only if
$(\beta_1,\beta_2)\ne(1/2,1/2)$.
\end{lemma}

\begin{proof}[Proof of Theorem~\ref{thm:transversal}]
Every $p$-quasirandom graph sequence satisfies the
labelled counting condition by
Lemma~\ref{lem:uniform-counting}.
If $\sum_{i=1}^m\alpha_i<1$, the conclusion follows
from \cite[Theorem~2.11(ii)]{JS}.
If $F=K_2$, the asserted classification is
Lemma~\ref{lem:edge-transversal}.
Hence assume that $\sum_{i=1}^m\alpha_i=1$ and
$F\ne K_2$, so $m\ge3$.

By Lemma~\ref{lem:transfer}, it suffices to show that
every graphon $W$ satisfying
\begin{equation}\label{eq:labelled-transversal-test}
\int_{A_1\times\cdots\times A_m}\Psi_{F,W}
=p^{e(F)}\prod_{i=1}^m\alpha_i
\end{equation}
for all pairwise disjoint measurable sets
$A_1,\ldots,A_m$ with $\lambda(A_i)=\alpha_i$
is equal to $p$ almost everywhere.

First suppose that $e(F)>1$, and put
$f=\Psi_{F,W}-p^{e(F)}$.
Then $f\in L^2([0,1]^m;\RR)$, and
\eqref{eq:labelled-transversal-test} gives
$\int_{A_1\times\cdots\times A_m}f=0$
for every admissible tuple.
\begin{revision}
By Lemma~\ref{lem:HHL} and the symmetrization calculation
following it, there is a function $h_0\in L^2([0,1];\RR)$
with $\int_0^1h_0=0$ such that
\[
\tPsi_{F,W}(x_1,\ldots,x_m)
=p^{e(F)}+\sum_{j=1}^m h_0(x_j)
\quad\text{almost everywhere}.
\]
\end{revision}
Set $h=h_0+p^{e(F)}/m$. Then
$h\in L^2([0,1])\subseteq L^1([0,1])$,
$\int_0^1h=p^{e(F)}/m>0$, and
$\tPsi_{F,W}=\sum_{j=1}^m h(x_j)$ almost everywhere.
Corollary~\ref{cor:additive} yields
$W=(m\int_0^1h)^{1/e(F)}=p$ almost everywhere.

It remains to consider $e(F)=1$.
Let $ab$ be the unique edge of $F$, so
$\Psi_{F,W}(x)=W(x_a,x_b)$.
Take arbitrary disjoint measurable sets $A_a,A_b$
with measures $\alpha_a,\alpha_b$.
Their complement has measure
$1-\alpha_a-\alpha_b
=\sum_{i\notin\{a,b\}}\alpha_i$.
By nonatomicity of Lebesgue measure, they extend
to a full pairwise disjoint tuple
$A_1,\ldots,A_m$ of the prescribed measures.
Fubini's theorem and
\eqref{eq:labelled-transversal-test} give
\[
\left(\prod_{i\notin\{a,b\}}\alpha_i\right)
\int_{A_a\times A_b}W(x,y)\,dx\,dy
=p\prod_{i=1}^m\alpha_i.
\]
Dividing by the positive product yields
$\int_{A_a\times A_b}(W-p)=0$ for every such pair.
Since $m\ge3$ and all remaining proportions are
positive, $\alpha_a+\alpha_b<1$.
Thus the labelled test for $F$ implies an edge test
with proportions whose sum is strictly less than one.
Applying the two-variable case of
Lemma~\ref{lem:HHL} to $W-p$ gives
$W=p$ almost everywhere.

Lemma~\ref{lem:transfer} now completes the proof.
\end{proof}

\begin{revision}
For $F=K_2$, the single-edge construction at the end of
Section~\ref{sec:repeated}, with $m=r=2$, gives a two-type
witness for the exceptional profile $(1/2,1/2)$, since
labelled and symmetrized edge counts coincide. Thus the
classification answers \cite[Problems~2.15 and~2.17]{JS}
affirmatively.
The distinction between labelled and symmetrized counts
is already visible for one edge and an isolated vertex:
every labelled transversal profile is forcing, whereas the
balanced symmetrized profile is not. We next consider
two repeated classes, where a higher Walsh component can
survive the restricted averages.
\end{revision}

\section{Two-class obstructions and counterexamples}\label{sec:counterexamples}

For the profile
$(\bm,\balpha)=((2,1),(1/3,2/3))$, the grouped polynomial
$E_2(z_1,z_2)=z_1(z_1+2z_2)=3z_1L_{\balpha}(z)$
has the prescribed linear divisor. Hence $2\in\cK$,
and the restricted averages annihilate $\HH_{3,2}$.
In this section, we prove Proposition~\ref{prop:crossing}
by constructing nonconstant graphons for every $p\in(0,1)$,
and then give a finite graph construction at $p=4^{-1/3}$.

For $F=K_3$, labelled and symmetrized counts coincide.
Let $G$ have $N$ vertices, and let $U,V\subseteq V(G)$
be disjoint with $|U|=\lfloor N/3\rfloor$ and
$|V|=\lfloor2N/3\rfloor$. Then $V\subseteq U^c$ and
$|U^c\setminus V|\le1$, so replacing $V$ by $U^c$
changes $N(K_3,G;U,U,V)$ by at most $N^2=o(N^3)$.
It therefore suffices to establish the graphon identities
over $A^2\times A^c$ with $\lambda(A)=1/3$.

\subsection{Counterexamples for every density parameter}

Let $\sigma:[0,1]\to\{-1,1\}$ be measurable, taking
each value on a set of measure $1/2$, and put
$Q_\sigma(x_1,x_2,x_3)
=\sum_{1\le i<j\le3}\sigma(x_i)\sigma(x_j)$.
Then $Q_\sigma\in\HH_{3,2}$ and
$\int_{A^2\times A^c}Q_\sigma=0$ whenever
$\lambda(A)=1/3$.
For each $p\in(0,1)$, the construction below gives
a nonconstant graphon $W$ satisfying
$\Psi_{K_3,W}-p^3=cQ_\sigma$ for some $c\ne0$
and $\int_{[0,1]^2}W<p$.
Approximating $W$ in cut distance by finite simple
graphs~\cite{LS} and applying
Lemma~\ref{lem:transfer} then yields the required
graph sequence.

\begin{proof}[Proof of Proposition~\ref{prop:crossing}]
Fix $0<p<1$. Let $B=[0,1/2)$ and
$\sigma=\mathbf{1}_B-\mathbf{1}_{B^c}$, so that
$\int_0^1\sigma=0$ and $\sigma^2=1$. Set
\[
b=\frac14\min\{p,1-p\},
\qquad
a=(p^3-b^3)^{1/3},
\qquad
W(x,y)=a+b\sigma(x)\sigma(y).
\]
Since $0<b\le p/4$, we have
$a^3=p^3-b^3\ge63b^3$, and hence $a>b>0$.
Moreover, $a<p$ and $b\le(1-p)/4$, so
$0<a-b\le W(x,y)\le a+b
<p+(1-p)/4<1$.
Thus $W$ is a graphon. It takes the value $a+b$
on $B^2\cup(B^c)^2$ and $a-b$ on
$(B\times B^c)\cup(B^c\times B)$, so it is
nonconstant and has two types. Its edge density is
$\int_{[0,1]^2}W=a+b(\int_0^1\sigma)^2=a<p$.

Define
$Q_\sigma(x_1,x_2,x_3)
=\sum_{1\le i<j\le3}\sigma(x_i)\sigma(x_j)$.
Writing $\sigma_i=\sigma(x_i)$ and using
$\sigma_i^2=1$, we obtain
\begin{align*}
\Psi_{K_3,W}(x_1,x_2,x_3)
&=(a+b\sigma_1\sigma_2)
  (a+b\sigma_1\sigma_3)
  (a+b\sigma_2\sigma_3)\\
&=a^3+b^3+(a^2b+ab^2)Q_\sigma(x_1,x_2,x_3)\\
&=p^3+ab(a+b)Q_\sigma(x_1,x_2,x_3).
\end{align*}
The function $g(x,y)=\sigma(x)\sigma(y)$ is symmetric
and has integral zero in each coordinate, so
$Q_\sigma\in\HH_{3,2}$. Moreover, $Q_\sigma=3$ on
$B^3$, which has positive measure. Since
$ab(a+b)>0$, the discrepancy
$\Psi_{K_3,W}-p^3$ is a nonzero element of
$\HH_{3,2}$.

Let $A\subseteq[0,1]$ be measurable, and put
$\alpha=\lambda(A)$ and $z=\int_A\sigma$.
Then $\int_{A^c}\sigma=-z$. Fubini's theorem gives
\begin{equation}\label{eq:pair-cancellation}
\int_{A^2\times A^c}Q_\sigma
=(1-\alpha)z^2-\alpha z^2-\alpha z^2
=(1-3\alpha)z^2.
\end{equation}
This is the special case
$g(x,y)=\sigma(x)\sigma(y)$ of
\cite[(9.4)]{JS}. In particular, if $\lambda(A)=1/3$,
then
\[
\int_{A^2\times A^c}\Psi_{K_3,W}
=p^3\lambda(A)^2\lambda(A^c).
\]
If $A,C$ are disjoint measurable sets of measures
$1/3$ and $2/3$, respectively, then $C=A^c$
up to a null set. Hence $W$ satisfies all the
prescribed integral identities for the profile
$((2,1),(1/3,2/3))$. Since $K_3$ is complete,
its labelled and symmetrized kernels coincide.

By the standard graphon approximation
theorem~\cite{LS}, choose finite simple graphs $G_n$
with $v_n=v(G_n)\to\infty$ and
$\delta_\square(W_{G_n},W)\to0$.
Lemma~\ref{lem:transfer} gives
\[
N(K_3,G_n;U,U,V)
=p^3|U|^2|V|+o(v_n^3)
\]
uniformly over all disjoint sets
$U,V\subseteq V(G_n)$ with
$|U|=\lfloor v_n/3\rfloor$ and
$|V|=\lfloor2v_n/3\rfloor$.
Since integration is invariant under measure-preserving
relabellings and the cut norm controls the integral
over $[0,1]^2$,
\[
\left|\frac{2e(G_n)}{v_n^2}-a\right|
=\left|\int_{[0,1]^2}W_{G_n}
       -\int_{[0,1]^2}W\right|
\le\delta_\square(W_{G_n},W)\longrightarrow0.
\]
Thus the limiting edge density is $a<p$, so
$(G_n)$ is not $p$-quasirandom.

In fact, $(G_n)$ is not quasirandom at any density.
If it were $q$-quasirandom for some $q\in[0,1]$,
its limiting edge density would give $q=a$.
On the other hand,
Lemma~\ref{lem:uniform-counting} and the prescribed
crossing-triangle counts would imply $q^3=p^3$,
since $|U|^2|V|=(2/27+o(1))v_n^3$.
Hence $q=p$, contradicting $a<p$.
\end{proof}

\begin{revision}
The construction realizes the cancellation in
\eqref{eq:pair-cancellation} as the triangle-kernel discrepancy
of a nonconstant graphon with edge density $a<p$.
Thus the prescribed triangle counts do not force the
edge density to equal the parameter $p$.
\end{revision}

For this profile,
$E_1=2z_1+z_2$,
$E_2=z_1(z_1+2z_2)$, and
$E_3=z_1^2z_2$.
Since $L_{\balpha}=(z_1+2z_2)/3$,
Theorem~\ref{thm:nullspace} gives $\cK=\{2\}$.
The discrepancy
$\Psi_{K_3,W}-p^3=ab(a+b)Q_\sigma$
is therefore a nonzero second-order Walsh component.
This realizes a quadratic obstruction and shows why
the additive reduction used in
Theorem~\ref{thm:repeated} does not extend to all
two-class profiles.

\subsection{Two disjoint cliques}

The preceding construction applies to every $p\in(0,1)$.
At $p=4^{-1/3}$, the same cancellation yields an explicit
sequence of finite graphs, with a uniform counting error
of order $v(G)^2$.

\begin{example}\label{ex:two-cliques}
For each positive integer $t$, let
$G_t=K_{3t}\sqcup K_{3t}$ and $p=4^{-1/3}$.
Write $L,R$ for the vertex sets of the two cliques.
For $U\subseteq V(G_t)$ with $|U|=2t$, put
$x=|U\cap L|$ and $y=|U\cap R|$, so $x+y=2t$.
Every triangle lies entirely in one clique, and hence
\begin{align*}
N(K_3,G_t;U,U,U^c)
&=x(x-1)(3t-x)+y(y-1)(3t-y)\\
&=4t^3-2t^2-2xy,
\end{align*}
where the second equality follows from
$x^2+y^2=4t^2-2xy$ and
$x^3+y^3=8t^3-6txy$.
Since $0\le xy\le t^2$ and
$p^3|U|^2|U^c|=4t^3$, the counting error is at most
$4t^2$, uniformly over $U$.
Every admissible pair of disjoint sets of sizes
$2t$ and $4t$ is complementary, so $(G_t)$ satisfies
the prescribed crossing-triangle condition.
However,
$2e(G_t)/v(G_t)^2=1/2-1/(6t)\to1/2<p$,
and therefore $(G_t)$ is not $p$-quasirandom.
\end{example}

\subsection{The remaining two-class problem}

For $0<p<1$ and $e(F)>0$,
Theorems~\ref{thm:repeated} and~\ref{thm:transversal}
classify symmetrized repeated-class tests with at least
three classes and labelled transversal tests, respectively.
For two repeated classes and $e(F)>1$, profiles with
$\alpha_1+\alpha_2<1$ are forcing by
Theorem~\ref{thm:nullspace} and
Corollary~\ref{cor:additive-profile}.
Thus the remaining classification concerns complementary
two-class profiles.

Both constructions above have limiting edge density
strictly less than $p$. They do not settle the variant
in which the limiting edge density is additionally
required to equal $p$.

\begin{problem}
Let $F$ be a simple graph on $m$ vertices with $e(F)>1$,
let $m_1,m_2$ be positive integers with $m_1+m_2=m$,
and let $0<p<1$. Determine all $\alpha\in(0,1)$ for which
$\tP^p_{m_1,m_2}(F;\alpha,1-\alpha)$ is a
$p$-quasirandom property.
\end{problem}

For a complementary profile, homogeneity and the
divisibility criterion in Theorem~\ref{thm:nullspace}
give
\[
\cK=
\left\{
k\in[m]:
R_{m_1,m_2,k}\!\left(-\frac{\alpha}{1-\alpha}\right)=0
\right\},
\]
where $R_{m_1,m_2,k}$ is defined in
Lemma~\ref{lem:app520-simple}.
Thus the functional nullspace is explicitly determined
by univariate polynomial zeros.
If $\cK\subseteq\{1\}$,
Corollary~\ref{cor:additive-profile} gives the forcing
property. When higher bad indices occur, the remaining
question is whether the corresponding nullspace contains
$\tPsi_{F,W}-p^{e(F)}$ for a nonconstant graphon $W$.
Proposition~\ref{prop:crossing} provides such a realization
for $F=K_3$ and $\cK=\{2\}$.

A further question is whether every nonforcing test in
this family admits a two-type graphon witness, as predicted
by \cite[Conjecture~6.6]{J26}.
All counting properties considered here concern
non-induced copies.

\section*{Declaration of generative AI and AI-assisted technologies}
The authors used \rev{a large language model} during the preparation of this manuscript
for assistance with mathematical exploration, checking arguments,
organization and language editing. The authors are responsible for
the mathematical statements, proofs, references and final text.

{\small
\noindent\rev{\textsc{Guoqing Wang}}\par
\noindent\textsc{School of Mathematics, Tiangong University}\par
\noindent\emph{Email address:} \texttt{gqwang1979@aliyun.com}\par
\medskip
\noindent\rev{\textsc{Linlin Wang}}\par
\noindent\textsc{School of Mathematics, China University of Mining and Technology}\par
\noindent\emph{Email address:} \texttt{5400@cumt.edu.cn}\par
\medskip
\noindent\rev{\textsc{Jun Yue}}\par
\noindent\textsc{School of Mathematics, Tiangong University}\par
\noindent\emph{Email address:} \texttt{yuejun06@126.com}\par
}

\end{document}